\documentclass[english,12pt,twoside]{article}
\usepackage{amssymb,amsbsy,amsmath,amsfonts,amssymb,amscd}
\usepackage{stmaryrd,floatflt}
\usepackage{latexsym}
\usepackage{euscript}
\usepackage{exscale}
\usepackage{epsfig}
\usepackage{amsthm}
\usepackage[french]{babel}

\newtheorem{theo}{Th\'eor\`eme}[section]

\newtheorem{lem}{Lemme}[section]
\newtheorem{prop}{Proposition}[section]

\newtheorem{cor}{Corollaire}[section]

\newcommand{\id}{\mathrm{id}}

\title{Compatibilit\'e des structures riemanniennes 
et des structures de Jacobi}
\author{Yacine A\"it Amrane, Ahmed Zeglaoui}

\begin{document}

\maketitle

\selectlanguage{french} \vspace*{2cm} \hspace{-0.6cm}{\bf R\'esum\'e.} 
On d\'efinit une notion de compatibilit\'e entre une structure riemannienne et une structure de Jacobi. On montre que dans le cas des structures de Poisson, des structures de contact et des structures  
localement conform\'ement symplectiques, des exemples fondamentaux de structures de Jacobi, on obtient respectivement des structures de Poisson riemanniennes au sens de M. Boucetta, des structures $\frac{1}{2}$-Kenmotsu et des structures localement conform\'ement k\"ahl\'eriennes. 
\bigskip\\
{\bf Mots Cl\'es.} Vari\'et\'es de Jacobi, de Poisson pseudo-riemanniennes,  riemanniennes de contact, riemanniennes presque de contact, de Kenmotsu, localement conform\'ement symplectiques, localement conform\'ement k\"ahl\'eriennes, alg\'ebro\"ides de Lie.\bigskip
 \bigskip\bigskip\\

\hspace{-0.6cm}{\bf\Large Introduction}\medskip\\

Les vari\'et\'es de Jacobi ont \'et\'e introduites s\'epar\'ement par  A. Lichnerowicz et A. Kirillov. Elles g\'en\'eralisent \`a la fois les vari\'et\'es de Poisson, les vari\'et\'es de contact et les vari\'et\'es localement conform\'ement symplectiques. On se pose la question naturelle de l'existence d'une notion de compatibilit\'e entre une structure de Jacobi et une structure pseudo-riemannienne,  qui pour des structures de Jacobi particuli\`eres, donne lieu \`a des structures g\'eom\'etriques remarquables. Dans ce travail, on introduit une telle notion qui dans le cas d'une vari\'et\'e de Poisson donne une structure de Poisson pseudo-riemannienne au sens de M. Boucetta. On montre que pour une structure de contact riemannienne, avec cette notion de compatibilit\'e on obtient une structure $\frac{1}{2}$-Kenmotsu, et que dans le cas d'une structure localement conform\'ement symplectique et d'une m\'etrique "associ\'ee", on retrouve une structure localement conform\'ement k\"ahl\'erienne.

Soit $M$ une vari\'et\'e diff\'erentiable. On consid\`ere sur $M$ un champ de bivecteurs $\pi$, un champ de vecteurs $\xi$ et une $1$-forme diff\'erentielle $\lambda$, et on associe au triplet $(\pi,\xi,\lambda)$ un alg\'ebro\"ide altern\'e $(T^\ast M,\sharp_{\pi,\xi},[.,.]_{\pi,\xi}^\lambda)$ sur $M$. On montre que si le  couple $(\pi,\xi)$ d\'efinit une structure Jacobi et que $\pi\neq 0$, l'alg\'ebro\"ide altern\'e $(T^\ast M,\sharp_{\pi,\xi},[.,.]_{\pi,\xi}^\lambda)$ est un pr\'ealg\'ebro\"ide de Lie si et seulement si $\sharp_{\pi,\xi}(\lambda)=\xi$. Dans le cas  $\xi=\lambda=0$, on retrouve l'alg\'ebro\"ide cotangent de la vari\'et\'e de Poisson $(M,\pi)$. On montre aussi que dans le cas o\`u $(\pi,\xi)$ est une structure de Jacobi associ\'ee \`a une forme de contact $\eta$, respectivement une structure localement  conform\'ement symplectique $(\omega,\theta)$, l'alg\'ebro\"ide altern\'e $(T^\ast
M,\sharp_{\pi,\xi},[.,.]_{\pi,\xi}^\eta)$, respectivement $(T^\ast
M,\sharp_{\pi,\xi},[.,.]_{\pi,\xi}^\theta)$, est un alg\'ebro\"ide de Lie isomorphe \`a l'alg\'ebro\"ide tangent $(TM,\textrm{Id}_M,[,])$ de $M$. 

Ensuite, pour un triplet $(\pi,\xi,g)$ form\'e d'un champ de bivecteurs $\pi$, un champ de vecteurs $\xi$ et une m\'etrique pseudo-riemannienne $g$ sur $M$, on pose $\lambda=g(\xi,\xi)\flat_g(\xi)-\flat_g(J\xi)$ et $\left[.,.\right] _{\pi ,\xi }^{g}=[.,.]_{\pi,\xi}^\lambda$, o\`u $\flat_g:TM\rightarrow T^\ast M$ et $\sharp_g=\flat_g^{-1}$ sont les isomorphismes musicaux de $g$, et $J$ est l'endomorphisme du fibr\'e tangent $TM$ donn\'e par $\pi(\alpha,\beta)=g(J\sharp_g(\alpha),\sharp_g(\beta))$, et on d\'efinit une d\'eriv\'ee contravariante $\mathcal{D}$ comme \'etant l'unique d\'eriv\'ee contravariante, sym\'etrique par rapport au crochet $\left[.,.\right] _{\pi ,\xi }^{g}$ et compatible avec $g$. Si $(\pi,\xi)$ est une structure de Jacobi, et si $\sharp_{\pi,\xi}$ est une isom\'etrie, une condition qui est particuli\`erement satisfaite dans  le cas d'une forme de contact et le cas structure localement conform\'ement symplectique, on montre que $\mathcal{D}$ est reli\'ee \`a la connexion de Levi-Civita $\nabla$ de $g$ par $\sharp_{\pi,\xi}(\mathcal{D}_\alpha\beta)=\nabla_{\sharp_{\pi,\xi}(\alpha)}\sharp_{\pi,\xi}(\beta)$.

Finalement, avec l'aide de la d\'eriv\'ee  de Levi-Civita contravariante $\mathcal{D}$ on introduit une notion de compatibilit\'e du triplet $(\pi,\xi,g)$. Dans le cas $\xi=0$, il s'agit juste de la compatibilit\'e du couple $(\pi,g)$ introduite par M. Boucetta, \cite{boucetta1}. Dans le cas d'une structure de Jacobi $(\pi,\xi)$ associ\'ee \`a une structure riemannienne de contact $(\eta,g)$, le triplet $(\pi,\xi,g)$ est compatible si et seulement si la structure $(\eta,g)$ est $\frac{1}{2}$-Kenmotsu. Dans le cas d'une structure de Jacobi $(\pi,\xi)$ associ\'ee \`a une structure localement conform\'ement symplectique $(\omega,\theta)$, si $g$ est une m\'etrique en quelque sorte associ\'ee, alors le triplet $(\pi,\xi,g)$ est compatible si et seulement si la structure $(\omega,\theta,g)$ est  localement conform\'ement k\"ahl\'erienne.

\section{Pr\'ealg\'ebro\"ides de Lie associ\'es \`a une vari\'et\'e de Jacobi}

\subsection{Pr\'ealg\'ebro\"ides de Lie associ\'es \`a une vari\'et\'e de
Jacobi}

Tout au long de ce travail $M$ d\'esigne une vari\'et\'e diff\'erentiable, $\pi$ un champ de bivecteurs et $\xi$ un champ de vecteurs sur $M$.   

Le couple $(\pi,\xi)$ d\'efinit une structure de Jacobi sur $M$ si on a les relations 
\begin{equation}\label{jacobi}
\begin{array}{ccc}
\left[ \pi ,\pi \right] =2\xi \wedge \pi & \text{ \ \ et \ \ } &
\left[ \xi
,\pi \right] :=\mathcal{L}_{\xi }\pi =0,%
\end{array}
\end{equation}
o\`u $\left[ .,.\right]$ d\'esigne le crochet de Schouten-Nijenhuis.
On dit que $(M,\pi,\xi)$ est une vari\'et\'e de Jacobi. Le cas $\xi
=0$, les relations ci-dessus \'etant r\'eduites \`a $\left[ \pi ,\pi
\right]=0$, correspond \`a une structure de Poisson $(M,\pi)$.

Rappelons qu'un alg\'ebro\"{\i}de altern\'e sur $M$ est un triplet $\left( E,\sharp _{E},\left[.,.\right]_{E}\right)$ o\`u $E$ est l'espace total d'un fibr\'e
vectoriel au-dessus de $M$, $\sharp_{E}$ est un morphisme de
fibr\'es vectoriels de $E$ dans $TM$, appel\'e l'application ancre,
et $\left[ .,.\right]_{E}:  \Gamma(E) \times \Gamma(E)\longrightarrow \Gamma(E)$, $(s,t)\longmapsto\left[ s,t\right]_{E}$, 
est une application $\mathbb{R}$-bilin\'eaire altern\'ee sur l'espace 
$\Gamma(E)$ des sections de $E$, v\'erifiant l'identit\'e
de Leibniz :
\begin{equation*}
\left[ s,\varphi t\right] _{E}=\varphi \left[ s,t\right] _{E}+
\sharp_{E}(s)(\varphi) t, \quad \forall \varphi \in C^{\infty}(M),\;
\forall s,t\in \Gamma(E).
\end{equation*}
Un alg\'ebro\"{\i}de altern\'e
$\left(E,\sharp_{E},\left[.,.\right]_{E}\right)$ est un
pr\'{e}alg\'ebro\"{\i}de de Lie si
\begin{equation*}
\sharp_{E}\left( \left[ s,t\right]_{E}\right) =\left[
\sharp_{E}(s),\sharp_{E}(t) \right] , \quad \forall s,t\in
\Gamma(E),
\end{equation*}
et un alg\'ebro\"{\i}de de Lie si $\left(\Gamma(E) ,\left[
.,.\right]_{E}\right) $ est une alg\`ebre de Lie, c'est-\`a-dire si
\begin{equation*}
\left[ s,\left[ t,r\right] _{E}\right] _{E}+\left[ t,\left[ r,s\right] _{E}%
\right] _{E}+\left[ r,\left[ s,t\right] _{E}\right] _{E}=0, \quad
\forall  s,t,r\in \Gamma(E).
\end{equation*}
Un alg\'ebro\"{\i}de de Lie est un pr\'ealg\'ebro\"{\i}de de Lie.
D'un autre c\^ot\'e, un pr\'ealg\'ebro\"{\i}de de Lie $\left(
E,\sharp_{E},\left[ .,.\right]_{E}\right)$ dont l'ancre $\sharp _{E}$ est un isomorphisme  est un
alg\'ebro\"{\i}de de Lie isomorphe \`a l'alg\'ebro\"{\i}de tangent
$(TM,\id_M,[.,.])$ de $M$.

Soit $\sharp_{\pi }: T^\ast M \longrightarrow TM$  le morphisme de fibr\'es vectoriels d\'efini par $
\beta\left(\sharp_{\pi }\left( \alpha \right)\right)=\pi\left(
\alpha ,\beta\right)$ et soit l'application $\left[ .,.\right]_{\pi }:\Omega^{1}(M) \times
\Omega ^{1}(M) \longrightarrow \Omega^{1}(M)$  d\'efinie par
\begin{equation*}\label{pi-koszul}
\left[ \alpha ,\beta \right] _{\pi }:=\mathcal{L}_{\sharp_{\pi
}(\alpha)}\beta -\mathcal{L}_{\sharp_{\pi}(\beta)}\alpha -d\left(
\pi(\alpha,\beta) \right),
\end{equation*}
appel\'ee le crochet de Koszul. Consid\'erons le morphisme de fibr\'es vectoriels
$\sharp_{\pi,\xi}:  T^{\ast}M  \longrightarrow  TM$ d\'efini par
\begin{equation*}
\sharp_{\pi,\xi}(\alpha)=\sharp_{\pi}(\alpha)+\alpha(\xi) \xi
\end{equation*}
et, pour une $1$-forme $\lambda \in \Omega ^{1}(M)$, l'application
$\left[ .,.\right]_{\pi ,\xi }^{\lambda}:  \Omega ^{1}(M)\times
\Omega ^{1}(M)\longrightarrow \Omega ^{1}(M)$ d\'efinie par
\begin{equation*}
\left[ \alpha ,\beta \right]_{\pi,\xi}^{\lambda}:=\left[
\alpha,\beta \right]_{\pi }+\alpha(\xi) \left(\mathcal{L}_{\xi}\beta
-\beta\right)-\beta(\xi) \left(\mathcal{L}_{\xi}\alpha-\alpha\right)
-\pi(\alpha,\beta) \lambda.
\end{equation*}
Le triplet $(T^{\ast}M,\sharp_{\pi,\xi},\left[.,.\right]_{\pi,\xi}^{\lambda})$, 
associ\'e \`a $(\pi,\xi,\lambda)$, est un alg\'ebro\"{\i}de altern\'e sur $M$.

Dans le cas o\`{u} $\xi =\lambda =0$, le triplet
$(T^{\ast}M,\sharp_{\pi,\xi },\left[ .,.\right]_{\pi,\xi
}^{\lambda})$ n'est rien d'autre que l'alg\'ebro\"{\i}de altern\'e
$(T^{\ast}M,\sharp_{\pi},\left[ .,.\right]_{\pi})$  associ\'e au
champ de bivecteurs $\pi$. Rappelons que quelles que soient les
formes diff\'erentielles $\alpha ,\beta ,\gamma \in \Omega ^{1}(M)$
on a
\begin{equation}\label{poisson-pre-alg}
\gamma\left(\sharp_\pi\left(\left[\alpha,\beta\right]_\pi\right)-\left[\sharp_\pi(\alpha),\sharp_\pi(\beta)\right]\right)
=\dfrac{1}{2}\left[\pi,\pi\right]\left(\alpha ,\beta,\gamma \right),
\end{equation}
et quelles que soient les fonctions $\varphi,\psi,\phi\in
C^\infty(M)$ on a
\begin{equation*}\label{poisson-alg}
\left[d\varphi,\left[d\psi,d\phi\right]_{\pi}\right]_{\pi}+\left[d\psi,\left[
d\phi,d\varphi\right]_{\pi}\right]_{\pi}+\left[d\phi,\left[d\varphi,d\psi
\right]_{\pi}\right]_{\pi}=-\dfrac{1}{2}d\left(\left[\pi,\pi\right]\left(d\varphi
,d\psi,d\phi \right)\right).
\end{equation*}
Ainsi, $(T^{\ast}M,\sharp_{\pi},\left[.,.\right]_{\pi })$ est un
alg\'ebro\"{\i}de de Lie si et seulement si $\pi$ est un tenseur de
Poisson. Si $\pi$ est un tenseur de Poisson sur $M$, le triplet $\left( T^{\ast
}M,\sharp _{\pi },\left[ .,.\right]_{\pi}\right)$ est appel\'e l'alg\'ebro\"{\i}de cotangent de la vari\'et\'e de
Poisson $(M,\pi)$. Dans le cas d'une structure de Jacobi on a le r\'esultat suivant

\begin{theo}\label{torsion-pre-alg}
Supposons que $(\pi,\xi)$ est une structure de Jacobi sur $M$ et soit
$\lambda\in \Omega^1(M)$. On a
$$
\sharp_{\pi,\xi}(\left[\alpha,\beta\right]_{\pi,\xi}^\lambda)-
\left[\sharp_{\pi,\xi}(\alpha),
\sharp_{\pi,\xi}(\beta)\right]=\pi(\alpha,\beta)\left(\xi-\sharp_{\pi,\xi}(\lambda)\right),
$$
quelles que soient les formes $\alpha,\beta\in \Omega^1(M)$.
\end{theo}

\begin{proof} On a d'une part
$$
\begin{array}{lll}
\sharp_{\pi,\xi}(\left[ \alpha ,\beta \right]_{\pi,\xi}^{\lambda})
&=&\sharp_{\pi}\left( \left[ \alpha ,\beta \right]_{\pi}\right)-\alpha(\xi)\sharp_\pi(\beta)+\beta(\xi)\sharp_\pi(\alpha)-\pi(\alpha,\beta)\sharp_{\pi,\xi}(\lambda)
+\alpha(\xi)\sharp_\pi(\mathcal{L}_\xi\beta)\\ 
&& -\beta(\xi)\sharp_\pi(\mathcal{L}_\xi
\alpha)+\left[\sharp_{\pi,\xi}(\alpha)(\beta(\xi))-\sharp_{\pi,\xi}(\beta)(\alpha(\xi))-\xi(\pi(\alpha,\beta))\right. \\ 
&& \left. +\beta(\mathcal{L}_\xi(\sharp_\pi(\alpha)))-\alpha(\mathcal{L}_\xi(\sharp_\pi(\beta)))\right]\xi
\end{array}
$$
et d'autre part
$$
\begin{array}{lll}
\left[ \sharp_{\pi,\xi}(\alpha),\sharp_{\pi,\xi}(\beta) \right] &=&\left[\sharp_{\pi}(\alpha),\sharp_{\pi}(\beta) \right] +\alpha(\xi)\mathcal{L}_\xi(\sharp_{\pi}(\beta)) -\beta(\xi)
\mathcal{L}_\xi(\sharp_{\pi}(\alpha))  \\ 
&&+\left[\sharp_{\pi,\xi}(\alpha)(\beta(\xi)) -
\sharp_{\pi,\xi}(\beta)(\alpha(\xi)) \right] \xi .
\end{array}
$$
Ainsi, en utilisant l'identit\'e (\ref{poisson-pre-alg}),  
on d\'eduit que
$$
\begin{array}{lll}
\sharp_{\pi,\xi}( \left[ \alpha ,\beta
\right]_{\pi,\xi}^{\lambda}) -\left[\sharp_{\pi,\xi}(\alpha)
,\sharp_{\pi,\xi}(\beta)\right]\!\!
&\!=\!&\!\!\left(\frac{1}{2}\left[\pi,\pi\right]-\xi\wedge
\pi\right)(\alpha,\beta,\cdot)-\alpha(\xi)\mathcal{L}_\xi\sharp_\pi(\beta)
+\beta(\xi)\mathcal{L}_\xi\sharp_\pi(\alpha) \\   
&& -\left[\alpha(\mathcal{L}_\xi\sharp_\pi(\beta))-\beta(\mathcal{L}_\xi\sharp_\pi(\alpha))+\mathcal{L}_\xi\pi(\alpha,\beta)\right]\xi\\  
&& +\pi(\alpha,\beta)(\xi-\sharp_{\pi,\xi}(\lambda)).
\end{array} 
$$
Il reste \`a utiliser les relations (\ref{jacobi}). 
\end{proof}

\begin{cor}\label{lambda-pre-alg}
Supposons que $(\pi,\xi)$ est une structure de Jacobi sur $M$ et soit
$\lambda\in\Omega^1(M)$. Si $\sharp_{\pi,\xi}(\lambda) =\xi$, alors
l'alg\'ebro\"{\i}de altern\'e $(T^{\ast}M,\sharp_{\pi ,\xi },\left[
.,.\right]_{\pi ,\xi}^{\lambda})$ associ\'e au triplet $(\pi,\xi
,\lambda)$ est un pr\'ealg\'ebro\"{\i}de de Lie, c'est-\`a-dire
$$
\sharp_{\pi,\xi}(\left[\alpha,\beta\right]_{\pi,\xi}^\lambda)=
\left[\sharp_{\pi,\xi}(\alpha), \sharp_{\pi,\xi}(\beta)\right],
$$
quelles que soient les formes $\alpha ,\beta \in \Omega ^{1}(M)$. La
r\'eciproque aussi est vraie si $\pi\neq 0$.
\end{cor}
\begin{proof}
D\'ecoule du th\'eor\`eme ci-dessus.
\end{proof}

\subsection{Alg\'ebro\"ide cotangent \`a une vari\'et\'e de contact}

Supposons $M$ de dimension impaire $2n+1$, $n\in \mathbb{N}^\ast$. Rappelons qu'une forme de contact sur $M$ est une $1$-forme diff\'erentielle $\eta$ sur $M$ telle que la forme $\eta \wedge (d\eta)^{\wedge^{n}}$ est une forme volume. Supposons que le couple $(\pi,\xi)$ est la structure de Jacobi associ\'ee \`a une forme de contact $\eta$ sur $M$, c'est-\`a-dire qu'on a
$$
\pi(\alpha,\beta)=d\eta\left(\sharp_{\eta }(\alpha)
,\sharp_{\eta}(\beta)\right),
$$
o\`u $\sharp_{\eta }$ est l'isomorphisme inverse de l'isomorphisme de fibr\'es vectoriels $\flat_{\eta }:  TM \rightarrow
T^\ast M$, $\flat_\eta(X)=-i_{X}d\eta +\eta(X) \eta$, et $\xi=\sharp_\eta(\eta)$. Le   champ $\xi$ est appel\'e le champ de Reeb associ\'e \`a la structure de contact $(M,\eta)$, il est caract\'eris\'e par les formules  
\begin{equation*}
\begin{array}{ccc}
i_{\xi}d\eta :=d\eta( \xi ,.) =0 & \text{ \ \ et \ \ } & i_{\xi}\eta
:=\eta(\xi) =1. 
\end{array}%
\end{equation*}

\begin{prop}\label{eta-alg}
L'alg\'ebro\"{\i}de altern\'e
$(T^{\ast}M,\sharp_{\pi ,\xi},\left[ .,.\right]_{\pi,\xi }^{\eta})$
est un alg\'ebro\"{\i}de de Lie isomorphe \`a l'alg\'ebro\"{\i}de
tangent de $M$.
\end{prop}
\begin{proof}
Montrons que $\sharp_{\pi,\xi}$ est \'egal \`a l'isomorphisme
$\sharp_\eta$, inverse de l'isomorphisme $\flat_\eta$. Soient
$\alpha,\beta \in \Omega^1(M)$, et soient $X,Y$ tels que
$\alpha=\flat_\eta(X)$ et $\beta=\flat_\eta(Y)$. Remarquons d'abord
qu'on a $\alpha(\xi)=\flat_\eta(X)(\xi)=\eta(X)$, 
et de m\^eme  $\beta(\xi)=\eta(Y)$. Ainsi, on a
$$
\beta(\sharp_{\pi,\xi}(\alpha))
=\pi(\alpha,\beta)+ \eta(X)\eta(Y) 
=(-i_Yd\eta+\eta(Y)\eta)(X)
=\flat_\eta(Y)(X)
=\beta(\sharp_\eta(\alpha)).
$$
Donc $\sharp_{\pi,\xi}=\sharp_{\eta}$ et en particulier
$\sharp_{\pi,\xi}(\eta)=\sharp_{\eta}(\eta)=\xi$. La proposition
d\'ecoule alors du corollaire \ref{lambda-pre-alg} et du fait que $\sharp_{\pi,\xi}$ est un isomorphisme. 
\end{proof}

Donc si $(M,\eta)$ est une vari\'et\'e de contact et si $(\pi,\xi)$
est la structure de Jacobi associ\'ee, d'apr\`es la proposition
ci-dessus, on a $\sharp_{\pi,\xi}=\sharp_\eta$. Si on pose $\left[
.,.\right]_\eta=\left[ .,.\right]_{\pi,\xi}^\eta$, alors on a un
alg\'ebro\"ide de Lie $(T^\ast M, \sharp_\eta,\left[
.,.\right]_\eta)$ associ\'e naturellement \`a la vari\'et\'e de
contact $(M,\eta)$.
 On pourra l'appeler l'alg\'ebro\"ide cotangent de la
vari\'et\'e de contact $(M,\eta)$.

\subsection{Alg\'ebro\"ide cotangent \`a une vari\'et\'e localement conform\'ement symplectique}

Une structure localement conform\'ement symplectique sur $M$ est un couple 
$(\omega,\theta)$ compos\'e d'une $1$-forme diff\'erentielle ferm\'ee $\theta$ et d'une
$2$-forme diff\'erentielle non d\'eg\'{e}n\'{e}r\'{e}e $\omega$ sur $M$ 
telles que
\begin{equation*}
d\omega +\theta \wedge \omega =0.
\end{equation*}
Dans le cas particulier o\`u la forme $\theta$ est exacte, i. e. 
$\theta =df$, on dit que $(\omega ,df)$ est conform\'ement
symplectique, ce qui est \'equivalent \`a $e^{f}\omega$ est symplectique, d'o\`u la  terminologie.

La proposition ci-dessous montre que la donn\'ee d'une vari\'et\'e
localement conform\'ement symplectique est \'equivalent \`a la
donn\'ee d'une vari\'et\'e de Jacobi dont le champ de bivecteurs
sous-jacent est non d\'eg\'en\'er\'e (voir aussi \cite[\S
2.3, exemple 4]{marle}). N'ayant pas trouv\'e de d\'emonstration dans la
litt\'erature, nous en donnons une ici.

Supposons que $\omega\in\Omega^2(M)$ est une $2$-forme non d\'eg\'en\'er\'ee et soit 
$\theta\in \Omega^1(M)$. Supposons que le couple $(\pi,\xi)$ est associ\'e au couple $(\omega,\theta)$, c'est-\`a-dire qu'on a  $
i_{\sharp_\pi(\alpha)}\omega=-\alpha$ pour tout $\alpha\in\Omega^1(M)$, et $i_\xi\omega = -\theta$. 
Autrement dit, 
$$
\pi(\alpha,\beta)=\omega(\sharp_\omega(\alpha),\sharp_\omega(\beta))
$$ 
o\`u $\sharp_\omega$ est l'isomorphisme inverse de l'isomorphisme de fibr\'es  vectoriels $\flat_\omega:TM \longrightarrow T^\ast M$, $\flat_\omega(X)=-i_X\omega$, et $\xi=\sharp_\omega(\theta)$.  Nous avons le r\'esultat suivant

\begin{lem}\label{omega-theta-pi-xi}
Quels que soient les champs de vecteurs $X,Y,Z\in
\mathfrak{X}(M)$, si $\alpha,\beta,\gamma\in \Omega^1(M)$ sont les
$1$-formes diff\'erentielles telles que $X=\sharp_\pi(\alpha)$,
$Y=\sharp_\pi(\beta)$ et $Z=\sharp_\pi(\gamma)$, alors on a
\begin{enumerate}
\item $(dw + \theta \wedge
\omega)(X,Y,Z)=\left(\frac{1}{2}\left[\pi,\pi\right]-\xi\wedge
\pi\right)(\alpha,\beta,\gamma)$.
\item $\mathcal{L}_\xi\omega(X,Y)=-\mathcal{L}_\xi\pi(\alpha,\beta)$.
\end{enumerate}
\end{lem}
\begin{proof} En utilisant l'identit\'e $\pi(\alpha,\beta)=\omega(X,Y)$ et l'identit\'e (\ref{poisson-pre-alg}), il vient 
$$
\omega([X,Y],Z)=\gamma([X,Y])=-\frac{1}{2}\left[\pi,\pi\right](\alpha,\beta,\gamma)+\pi(\left[\alpha,\beta\right]_\pi,\gamma),
$$
d'o\`u, avec un calcul direct, l'on d\'eduit  
\begin{equation}\label{domega}
d\omega(X,Y,Z)=\frac{1}{2}\left[\pi,\pi\right](\alpha,\beta,\gamma).
\end{equation}
Par ailleurs, remarquons que
$\theta(X)=-i_\xi\omega(X)=i_X\omega(\xi)=i_{\sharp_\pi(\alpha)}\omega(\xi)=-\alpha(
\xi)$, de m\^eme $\theta(Y)=-\beta(\xi)$ et
$\theta(Z)=-\gamma(\xi)$, d'o\`u $\theta\wedge\omega(X,Y,Z)=-\xi\wedge\pi(\alpha,\beta,\gamma)$.
D'o\`u, avec (\ref{domega}), la premi\`ere assertion du lemme. Pour la deuxi\`eme assertion, il suffit de remarquer qu'on a 
$$
\pi(\mathcal{L}_\xi\alpha,\beta)=-\mathcal{L}_\xi\alpha(Y)
=-\xi(\alpha(Y))+\alpha(\mathcal{L}_{\xi}Y)=\xi(\omega(X,Y))-\omega(X,\mathcal{L}_\xi
Y).
$$
\end{proof}

\begin{prop}\label{loc-conf-sympl-jac}
Le couple $(\omega,\theta)$ est une structure localement conform\'ement symplectique si et seulement si le couple $(\pi,\xi)$ est une structure de Jacobi. 
\end{prop}
\begin{proof}
De l'assertion 1. du lemme \ref{omega-theta-pi-xi} on d\'eduit que l'identi\'e
$d\omega+\theta \wedge \omega =0$ est satisfaite si et seulement
l'identit\'e $\left[ \pi ,\pi \right] =2\xi \wedge \pi$ l'est, et si
l'une des deux est satisfaite alors, en utilisant la formule de
Cartan, on obtient que
$$
\mathcal{L}_\xi\omega  =d(i_\xi \omega)+i_\xi d\omega 
=-d\theta - i_\xi(\theta\wedge\omega) 
=-d\theta,
$$
donc, avec l'assertion 2. du lemme \ref{omega-theta-pi-xi},  que
$\mathcal{L}_\xi\pi=0$ si et seulement si $d\theta=0$.
\end{proof}

\begin{prop}\label{omega-theta-alg}
Supposons que $(M,\omega,\theta)$ est une vari\'et\'e localement conform\'ement
symplectique et soit $(\pi,\xi)$ la structure de Jacobi associ\'ee. 
L'alg\'ebro\"ide altern\'e 
$(T^{\ast}M,\sharp_{\pi,\xi},\left[ .,.\right]_{\pi,\xi }^{\theta})$ est un
alg\'ebro\"{\i}de de Lie isomorphe \`a l'alg\'ebro\"{\i}de tangent
de $M$.
\end{prop}
\begin{proof}
Comme $\sharp_{\pi,\xi}(\theta)=\sharp_\pi(\theta)+\theta(\xi)\xi=\sharp_\pi(\theta)=\xi$. Alors, d'apr\`es le corollaire \ref{lambda-pre-alg}, le triplet $(T^{\ast}M,\sharp_{\pi ,\xi},\left[ .,.\right]_{\pi,\xi
}^{\theta})$ est un pr\'ealg\'ebro\"{\i}de de Lie. Il reste \`a montrer que $\sharp_{\pi,\xi}$ est un isomorphisme. 
Il suffit de montrer qu'il est injectif. Comme on a $\sharp_{\pi,\xi}(\alpha)=\sharp_\pi(\alpha+\alpha(\xi)\theta)$ et le champ $\pi$ est non d\'eg\'en\'er\'e, alors $\sharp_{\pi,\xi}(\alpha)=0$ entra\^ine $\alpha=-\alpha(\xi)\theta$, donc $\alpha(\xi)=-\alpha(\xi)\theta(\xi)=0$, donc $\alpha=0$. 
\end{proof}

D'o\`u, si $(M,\omega,\theta)$ est une vari\'et\'e localement
conform\'ement symplectique et si $(\pi,\xi)$ est la structure de
Jacobi associ\'ee, d'apr\`es la proposition ci-dessus, si on pose
$\sharp_{\omega,\theta}:=\sharp_{\pi,\xi}$
et $\left[ .,.\right]_{\omega,\theta}:=\left[
.,.\right]_{\pi,\xi}^\theta$, alors on a un alg\'ebro\"ide de Lie
$(T^\ast M, \sharp_{\omega,\theta},\left[
.,.\right]_{\omega,\theta})$ associ\'e naturellement \`a la
vari\'et\'e localement conform\'ement symplectique
$(M,\omega,\theta)$. On pourra l'appeler l'alg\'ebro\"ide cotangent
de la vari\'et\'e localement conform\'ement symplectique
$(M,\omega,\theta)$.

\section{D\'eriv\'ee de Levi-Civita contravariante associ\'ee au triplet
$(\pi,\xi,g)$}

\subsection{D\'efinition et propri\'et\'es}

Dans toute la suite, on d\'esigne par $g$ une m\'etrique pseudo-riemannienne sur $M$, par $\flat_g:TM\rightarrow T^\ast M$ l'isomorphisme de fibr\'es vectoriels tel que
$\flat_g(X)(Y)=g(X,Y)$, par $\sharp_{g}$ l'isomorphisme inverse de
$\flat_g$, et par $g^\ast$ la com\'etrique de $g$, c'est-\`a-dire le
champ de tenseurs d\'efini par $
g^\ast(\alpha,\beta) :=g\left( \sharp_{g}(\alpha) ,\sharp_{g}(\beta)
\right)$. 

Au couple $(\pi,g)$ on associe les champs d'endomorphismes $J$ de $TM$ et $J^\ast$ de $T^\ast M$ d\'efinis respectivement par
\begin{equation}\label{J}
g(J\sharp_{g}(\alpha),\sharp_{g}(\beta))=\pi(\alpha,\beta) \quad\textrm{ et }\quad g^\ast(J^\ast \alpha ,\beta) =\pi(\alpha,\beta). 
\end{equation}
On a $J=\sharp_g\circ J^\ast \circ \flat_g$. Au triplet $(\pi,\xi,g)$ on associe la $1$-forme diff\'erentielle
$\lambda$ d\'efinie par
\begin{equation*}
\lambda =g(\xi,\xi)\flat_{g}(\xi)-\flat_g(J\xi),
\end{equation*}
et on note $\left[.,.\right] _{\pi ,\xi }^{g}$ au lieu de $\left[
.,.\right]_{\pi,\xi}^{\lambda}$. 

On appelle la d\'eriv\'ee de Levi-Civita contravariante associ\'ee
au triplet $(\pi,\xi,g)$ d'unique d\'eriv\'ee $\mathcal{D}:
\Omega ^{1}(M)\times \Omega ^{1}(M) \longrightarrow \Omega ^{1}(M)$, 
sym\'etrique par rapport au crochet $\left[.,.\right] _{\pi ,\xi }^{g}$ et compatible avec la m\'etrique. Elle est enti\`erement caract\'eris\'ee par la formule :
\begin{equation}\label{formule-koszul-jacobi}
\begin{array}{lll}
2g^\ast\left( \mathcal{D}_{\alpha }\beta,\gamma \right)
&=&\sharp_{\pi,\xi}(\alpha)\cdot g^\ast(\beta,\gamma)
+\sharp_{\pi,\xi}(\beta)\cdot
g^\ast(\alpha,\gamma)-\sharp_{\pi,\xi}(\gamma)\cdot g^\ast(\alpha,\beta) \\
&&-g^\ast(\left[\beta,\gamma\right]_{\pi,\xi }^{g},\alpha)
-g^\ast(\left[\alpha,\gamma \right]_{\pi,\xi }^{g},\beta)
+g^\ast(\left[\alpha,\beta\right]_{\pi,\xi }^{g},\gamma).
\end{array}
\end{equation}
Dans le cas o\`u $\xi =0$, la d\'eriv\'ee $\mathcal{D}$ n'est rien
d'autre que la d\'eriv\'ee de Levi-Civita contravariante 
associ\'ee dans \cite{boucetta1} au couple $(\pi,g)$.

\begin{prop}\label{Levi-Civita-triplet}
Sopposons que l'alg\'ebro\"ide altern\'e $(T^\ast M,\sharp_{\pi,\xi},\left[
.,.\right] _{\pi ,\xi }^{g})$ est un pr\'ealg\'ebro\"ide de Lie et que l'application ancre $\sharp_{\pi,\xi}$ est une isom\'etrie. Alors
$$
\sharp_{\pi,\xi}\left(\mathcal{D}_\alpha\beta\right)
=\nabla_{\sharp_{\pi,\xi}(\alpha)}\sharp_{\pi,\xi}(\beta).
$$
o\`u $\nabla$ est la connexion de Levi-Civita (covariante)
associ\'ee \`a $g$.
\end{prop}
\begin{proof}
Puisque on a suppos\'e que $(T^\ast M,\sharp_{\pi,\xi},\left[
.,.\right] _{\pi ,\xi }^{g})$ est un pr\'ealg\'ebro\"ide de Lie, on a  
$$
\sharp_{\pi,\xi}(\left[\alpha,\beta\right]_{\pi,\xi}^g)=\left[\sharp_{\pi,\xi}(\alpha),\sharp_{\pi,\xi}(\beta)\right],
$$
quelles que soient les formes $\alpha,\beta \in \Omega^1(M)$. Comme
on a suppos\'e aussi que $\sharp_{\pi,\xi}$ est une isom\'etrie, de la
formule (\ref{formule-koszul-jacobi}) et de la formule de Koszul relative
\`a la connexion de Levi-Civita $\nabla$ de $g$ on d\'eduit que 
$$
g^\ast\left( \sharp_{\pi,\xi}(\mathcal{D}_{\alpha
}\beta),\sharp_{\pi,\xi}(\gamma) \right)= g\left(
\nabla_{\sharp_{\pi,\xi}(\alpha)
}\sharp_{\pi,\xi}(\beta),\sharp_{\pi,\xi}(\gamma) \right)
$$
quelles que soient les formes $\alpha,\beta,\gamma \in \Omega^1(M)$.
\end{proof}

\subsection{Alg\'ebro\"ide altern\'e associ\'e \`a une vari\'et\'e riemannienne presque de contact}

Soit $(\Phi,\xi,\eta)$ un triplet compos\'e d'une
$1$-forme $\eta$, d'un champ de vecteurs $\xi$ et d'un champ de
$(1,1)$-tenseurs $\Phi$ sur $M$. Le triplet $(\Phi,\xi,\eta)$ d\'efinit une structure presque de contact sur $M$ si $\Phi^{2}=-\textrm{Id}_{TM}+\eta \otimes \xi$ et $\eta(\xi)=1$. Il en r\'esulte, voir par exemple \cite[Th. 4.1]{blair}, que $\Phi(\xi)=0$ et $\eta \circ \Phi =0$.

On dit que la m\'etrique $g$ est associ\'ee au triplet $(\Phi,\xi,\eta)$ si l'identit\'e suivante est v\'erifi\'ee 
\begin{equation}\label{met-ass-presque-cont}
g\left(\Phi(X),\Phi(Y)\right)=g(X,Y)-\eta(X)\eta(Y).
\end{equation}
On dit que la vari\'et\'e $(M,\Phi,\xi,\eta,g)$ est une vari\'et\'e pseudo-riemannienne presque de contact si le triplet $(\Phi,\xi,\eta)$ est une structure presque de contact et que $g$ est une m\'etrique associ\'ee. Si de plus la m\'etrique $g$ est d\'efinie positive, on dit que $(M,\Phi,\xi,\eta,g)$ est une
vari\'et\'e riemannienne presque de contact. Remarquons que si on
met $Y=\xi$ dans la formule (\ref{met-ass-presque-cont}), on
d\'eduit que si $(\Phi,\xi,\eta,g)$ est une structure pseudo-riemannienne presque de contact alors
\begin{equation*}
g(X,\xi)=\eta(X),
\end{equation*}
pour tout $X\in \mathfrak{X}(M)$, i.e. $\flat_g(\xi)=\eta$. En particulier,
$g(\xi,\xi)=1$.

\begin{prop}\label{pi-ass-presque-contact}
Supposons que $(\Phi,\xi,\eta,g)$ est une structure pseudo-riemannienne
presque de contact sur $M$. L'application $\pi:\Omega^{1}(M)\times
\Omega^{1}(M)\rightarrow C^\infty(M)$ d\'efinie par
$$
\pi(\alpha,\beta)=g(\sharp_g(\alpha),\Phi(\sharp_g(\beta)))
$$
est un champ de bivecteurs sur $M$ et le morphisme de fibr\'es
$\sharp_{\pi,\xi}$ est une isom\'etrie.
\end{prop}

\begin{proof}
Soit $\alpha\in\Omega^1(M)$ et posons $X=\sharp_g(\alpha)$. En
utilisant (\ref{met-ass-presque-cont}) et $\eta\circ\Phi=0$, il
vient que $\pi(\alpha,\alpha)= g(\Phi(X),\Phi^2(X))$, 
et comme $\Phi^2=-\textrm{Id}_{TM}+\eta\otimes\xi$ et que 
$g(\Phi(X),\xi)=\eta\circ\Phi(X)=0$, il vient que  
$$
\pi(\alpha,\alpha)=-g(\Phi(X),X)+\eta(X)g(\Phi(X),\xi) =
-g(\Phi(X),X) = -\pi(\alpha,\alpha),
 $$
et donc, que $\pi$ est un champ de bivecteurs.
Montrons que $\sharp_{\pi,\xi}$ est une isom\'etrie.
 Soit $\alpha\in \Omega^1(M)$.
 Rappelons que par d\'efinition, on a
 $\sharp_{\pi,\xi}(\alpha)=\sharp_\pi(\alpha)+\alpha(\xi)\xi$.
 Comme on a d'un c\^ot\'e
 $\alpha(\xi)=g(\sharp_g(\alpha),\xi)=\eta(\sharp_g(\alpha))$ et
 d'un autre, pour tout $\beta\in \Omega^1(M)$,
 $$
\beta(\sharp_\pi(\alpha))=\pi(\alpha,\beta)=g(\sharp_g(\alpha),\Phi(\sharp_g(\beta)))
=-g(\Phi(\sharp_g(\alpha)),\sharp_g(\beta))=-\beta(\Phi(\sharp_g(\alpha))),
 $$
c'est-\`a-dire $\sharp_\pi(\alpha)=-\Phi(\sharp_g(\alpha))$, on
d\'eduit que
\begin{equation}\label{f1}
\sharp_{\pi,\xi}(\alpha)=-\Phi(\sharp_g(\alpha))+\eta(\sharp_g(\alpha))\xi.
\end{equation}
Soient $\alpha,\beta\in \Omega^1(M)$. De la formule (\ref{f1}) et du
fait qu'on a $g(\Phi(X),\xi)=\eta\circ\Phi(X)=0$ et $g(\xi,\xi)=1$,
on d\'eduit que
$$
g(\sharp_{\pi,\xi}(\alpha),\sharp_{\pi,\xi}(\beta))
=g(\Phi(\sharp_g(\alpha)),\Phi(\sharp_g(\beta)))
+\eta(\sharp_g(\alpha))\eta(\sharp_g(\beta)).
$$
En utilisant la formule (\ref{met-ass-presque-cont}), on obtient $
g(\sharp_{\pi,\xi}(\alpha),\sharp_{\pi,\xi}(\beta))
=g(\sharp_g(\alpha),\sharp_g(\beta))=g^\ast(\alpha,\beta). $
\end{proof}

\begin{cor}\label{riem-presque-contact-levi-civita}
Supposons que $(\Phi,\xi,\eta,g)$ est une structure pseudo-riemannienne
presque de contact sur $M$ et soit $\pi$ le champ de bivecteurs associ\'e,
c'est-\`a-dire d\'efini dans la proposition
\ref{pi-ass-presque-contact}. Si l'alg\'ebro\"ide altern\'e $(T^\ast M,\sharp_{\pi,\xi},\left[.,.\right] _{\pi ,\xi }^{g})$ est un pr\'ealg\'ebro\"ide de Lie, alors
\begin{equation*}
\sharp_{\pi,\xi}\left( \mathcal{D}_{\alpha }\beta \right)
=\mathcal{\nabla}_{\sharp_{\pi,\xi}(\alpha)}\sharp_{\pi,\xi}(\beta),
\end{equation*}
pour tous $\alpha ,\beta \in \Omega ^{1}(M)$.
\end{cor}
\begin{proof}
C'est une cons\'equence directe des propositions 
\ref{pi-ass-presque-contact} et \ref{Levi-Civita-triplet}.
\end{proof}

Supposons que $\eta$ est une forme de contact sur $M$. La vari\'et\'e $(M,\eta,g)$ est dite pseudo-riemannienne de contact, ou que $g$ est associ\'ee \`a la forme de contact $\eta$, s'il existe un champ d'endomorphismes $\Phi$ de $TM$ tel que
$(\Phi,\xi,\eta,g)$ est une structure pseudo-riemannienne presque de
contact et que
\begin{equation}\label{met-ass-cont}
g(X,\Phi(Y))=d\eta(X,Y).
\end{equation}
Si de plus $g$ est d\'efinie positive, on dit que
$(M,\eta,g)$ est une vari\'et\'e riemannienne de contact.

\begin{theo}\label{riemannienne-contact-levi-civita}
Supposons que $(M,\eta,g)$ est une vari\'et\'e pseudo-riemannienne de contact. On a
\begin{equation*}
\sharp_{\eta}\left( \mathcal{D}_{\alpha }\beta \right)
=\mathcal{\nabla}_{\sharp_{\eta}(\alpha)}\sharp_{\eta}(\beta),
\end{equation*}
pour tous $\alpha ,\beta \in \Omega ^{1}(M)$.
\end{theo}
\begin{proof}
Supposons donc $(M,\eta,g)$ pseudo-riemannienne de contact et 
soit $(\Phi,\xi,\eta,g)$ la structure pseudo-riemannienne presque de
contact associ\'ee.  Soit $(\pi,\xi)$ la structure de Jacobi
associ\'ee \`a $\eta$, alors $\sharp_\eta=\sharp_{\pi,\xi}$. D'apr\`es la proposition \ref{eta-alg} et le corollaire ci-dessus, il nous suffit de montrer que $\pi$ est associ\'e \`a $(\Phi,\xi,\eta,g)$ et que $\eta=\lambda$. Soit  $\alpha \in
\Omega^1(M)$ et posons $X=\sharp_\eta(\alpha)$. En utilisant (\ref{met-ass-cont}),  on a
$$
\sharp_g(\alpha)=\sharp_g(\flat_\eta(X))=-\sharp_g(i_X
d\eta)+\eta(X)\xi = \Phi(X)+\eta(X)\xi. 
$$
D'o\`u, en appliquant $\Phi$, 
$$
\Phi(\sharp_g(\alpha))=\Phi^2(X)=-X+\eta(X)\xi=-\sharp_\eta(\alpha)+\alpha(\xi)\xi=-\sharp_\pi(\alpha). 
$$
D'o\`u l'on d\'eduit que $\pi(\alpha,\beta)=g(\sharp_g(\alpha),\Phi(\sharp_g(\beta)))$ pour
tous $\alpha,\beta\in \Omega^1(M)$, et que $\Phi=-J$, o\`u $J$ est le champ d'endomorphismes associ\'e au couple $(\pi,g)$. Ainsi, $J\xi=0$, et comme $g(\xi,\xi)=1$, il vient que $\lambda=\flat_g(\xi)=\eta$.
\end{proof}

\subsection{M\'etrique riemannienne associ\'ee \`a une structure localement conform\'ement symplectique}

Supposons que $\omega\in \Omega^2(M)$ est une $2$-forme non d\'eg\'en\'er\'ee et
soit $\theta\in\Omega^1(M)$. Supposons que le couple $(\pi,\xi)$ est associ\'e au couple $(\omega,\theta)$. On dit que la m\'etrique
pseudo-riemmannienne $g$ est associ\'ee au couple
$(\omega,\theta)$ si $\sharp_{\omega,\theta}:=\sharp_{\pi,\xi}$ est une isom\'etrie, c'est-\`a-dire si 
\begin{equation}\label{met-ass-loc-conf-symp}
g\left(\sharp_{\omega,\theta}(\alpha),\sharp_{\omega,\theta}(\beta)\right)
=g^\ast(\alpha,\beta),
\end{equation}
pour tous $\alpha ,\beta \in \Omega ^{1}(M)$.

Si $\theta =0$, alors $\xi=0$ et
$\sharp_{\omega,\theta}=\sharp_{\omega}$, et si $J$ et $J^\ast$ sont les champs d'endomorphismes d\'efinis par les
formules (\ref{J}), alors 
$$
g(\sharp_{\omega,\theta}(\alpha),\sharp_{\omega,\theta}(\beta))=g(\sharp_{\omega}(\alpha),\sharp_{\omega}(\beta))=g^\ast(\flat_g(\sharp_{\omega}(\alpha)),\flat_g(\sharp_{\omega}(\beta)))=g^\ast(J^\ast\alpha,J^\ast\beta),
$$
pour tous $\alpha ,\beta \in \Omega ^{1}(M)$. Ainsi, dans le cas
$\theta=0$, la relation (\ref{met-ass-loc-conf-symp}) est \'equivalente \`a
\begin{equation*}
g^\ast\left( J^\ast\alpha ,J^\ast\beta \right) =
g^\ast(\alpha,\beta).
\end{equation*}
Si de plus $g$ est d\'efinie positive, cette derni\`ere identit\'e signifie que le couple $(\omega,g)$ est une structure presque hermitienne sur $M$ et que $J$ est la  structure presque complexe associ\'ee, c'est-\`a-dire, on a 
$$
g(JX,JY)=g(X,Y) \qquad\textrm{ et }\qquad \omega(X,Y)=g(JX,Y), 
$$
pour tous $X,Y\in\mathfrak{X}(M)$.

\begin{theo}\label{loc-conf-symp-met-ass-levi-civita}
Supposons que $(\omega,\theta)$ est une structure localement conform\'ement symplectique et que $g$ est une m\'etrique associ\'ee. On a 
$$
\sharp_{\omega,\theta}\left(\mathcal{D}_\alpha\beta\right)=\nabla_{\sharp_{\omega,\theta}(\alpha)}\sharp_{\omega,\theta}(\beta)
$$ 
pour tous $\alpha,\beta\in \Omega^1(M)$. 
\end{theo} 
\begin{proof}
D'apr\`es les propositions \ref{Levi-Civita-triplet} et \ref{omega-theta-alg}, il suffit de montrer que $\lambda=\theta$. D'un c\^ot\'e, on a $
\sharp_{\pi,\xi}(\theta)=\xi$. D'un autre c\^ot\'e, pour tout  $\alpha\in\Omega^1(M)$, on a
$$
\begin{array}{ll}
g(\sharp_{\pi,\xi}(\lambda),\sharp_{\pi,\xi}(\alpha))&
=g(\sharp_g(\lambda),\sharp_g(\alpha))\\
 & =g(\xi,\xi)\alpha(\xi)+g(\xi,J\sharp_g(\alpha)) \\
 & =g(\xi,\xi)\alpha(\xi)+g(\xi,\sharp_\pi(\alpha)) \\
 & =g(\xi,\sharp_{\pi,\xi}(\alpha)).
\end{array}
$$
Comme $\sharp_{\pi,\xi}$ est une isom\'etrie, donc un isomorphisme,
alors $\sharp_{\pi,\xi}(\lambda)=\xi$. 
\end{proof}

\begin{cor}\label{levi-civita-omega}
Sous les m\^emes hypoth\`eses que le th\'eor\`eme ci-dessus. On a 
$$
\mathcal{D}\pi(\alpha,\beta,\gamma)=\nabla \omega(\sharp_{\omega,\theta}(\alpha),\sharp_{\omega,\theta}(\beta),\sharp_{\omega,\theta}(\gamma)).
$$
\end{cor} 
\begin{proof}
On a $\omega \left( \xi
,\sharp_{\pi}(\alpha) \right) =- i_{\sharp_{\pi}(\alpha)}\omega(\xi)
=\alpha(\xi)$ et de m\^eme
$\omega(\xi,\sharp_\pi(\beta))=\beta(\xi)$, par cons\'equent 
\begin{equation}\label{omega-pi-xi}
\omega(\sharp_{\pi,\xi}(\alpha) ,\sharp_{\pi,\xi}(\beta))=\pi(\alpha,\beta).
\end{equation} 
Il suffit maintenant de calculer $\nabla \omega(\sharp_{\pi,\xi}(\alpha),\sharp_{\pi,\xi}(\beta),\sharp_{\pi,\xi}(\gamma)) $ et d'utiliser le th\'eor\`eme ci-dessus.
\end{proof}

\section{Compatibilit\'e du triplet $(\pi,\xi,g)$}

\subsection{D\'efinition}

On dit que $g$ est compatible avec le couple $(\pi,\xi)$ ou que le
triplet $(\pi,\xi,g)$ est compatible si
\begin{equation}\label{COMPATIBILITE2}
\mathcal{D}\pi(\alpha,\beta,\gamma)=\frac{1}{2}\left( \gamma(\xi)
\pi(\alpha ,\beta)-\beta(\xi) \pi(\alpha,\gamma) - J^\ast\gamma(\xi)
g^\ast(\alpha,\beta) + J^\ast\beta(\xi) g^\ast(\alpha,\gamma)
\right),
\end{equation}
pour tous $\alpha ,\beta,\gamma \in \Omega ^{1}(M)$. La formule (\ref{COMPATIBILITE2}) peut aussi s'\'ecrire sous la forme
\begin{equation}\label{COMPATIBILITE}
\left( \mathcal{D}_{\alpha }J^\ast\right) \beta =\frac{1}{2}\left(
\pi(\alpha,\beta)\flat_g(\xi) - \beta(\xi)J^\ast\alpha +
g^\ast(\alpha,\beta) J^\ast\flat_g(\xi) + J^\ast\beta(\xi)\alpha
\right),
\end{equation}
pour tous $\alpha ,\beta \in \Omega ^{1}(M)$.

La compatibilt\'e dans le cas o\`u le champ $\xi$ est nul signifie
que $(M,\pi,g)$ est une vari\'et\'e de Poisson pseudo-riemannienne,
et de Poisson riemannienne si de plus la m\'etrique $g$ est
d\'efinie positive.

\subsection{Vari\'et\'es $\frac{1}{2}$-Kenmotsu}

Rappelons, voir par exemple \cite[\S\,6.6]{blair}, qu'une structure riemannienne presque de contact $(\Phi,\xi,\eta,g)$ sur $M$ est dite $\frac{1}{2}$-Kenmotsu si 
\begin{equation*}
\left( \nabla _{X}\Phi \right)(Y) = \displaystyle\frac{1}{2}\left( g(\Phi(X),Y) \xi -\eta(Y)\Phi(X)\right),
\end{equation*}
pour tous $X,Y\in \mathfrak{X}(M)$.

\begin{lem}\label{phi-sharp-pi}
Supposons que $(\Phi,\xi,\eta,g)$ est une structure pseudo-riemannienne
presque de contact sur $M$ et soit $\pi$ le champ de bivecteurs associ\'e. Si l'alg\'ebro\"ide altern\'e $(T^\ast M,\sharp_{\pi,\xi},\left[.,.\right] _{\pi ,\xi }^{g})$ est un pr\'ealg\'ebro\"ide de Lie, alors 
$$
\sharp_{\pi,\xi}\left((\mathcal{D}_{\alpha}
J^\ast)\beta\right)=-\left(\nabla_{\sharp_{\pi,\xi}(\alpha)}\Phi\right)(\sharp_{\pi,\xi}(\beta)),
$$
pour tous $\alpha,\beta\in\Omega^1(M)$.
\end{lem}
\begin{proof}
En  utilisant la formule (\ref{f1}) et le fait qu'on a $\sharp_g\circ J^\ast=J\circ \sharp_g$, il vient que
$
\sharp_{\pi,\xi}(J^\ast\alpha) = -\Phi(\sharp_g(J^\ast\alpha))+
\eta(\sharp_g(J^\ast\alpha))\xi=-\Phi(\sharp_\pi(\alpha))=-\Phi(\sharp_{\pi,\xi}(\alpha)).
$
Donc 
\begin{equation}\label{pi-xi-J-ast-Phi}
\sharp_{\pi,\xi}\circ J^\ast=-\Phi \circ \sharp_{\pi,\xi}.
\end{equation}
D'o\`u, avec le corollaire
\ref{riem-presque-contact-levi-civita}, on a
\begin{eqnarray*}
\sharp_{\pi,\xi}\left( \left( \mathcal{D}_{\alpha }J^\ast\right)
\beta \right) &=&\sharp_{\pi,\xi}\left( \mathcal{D}_{\alpha }\left(
J^\ast\beta \right) \right) -\left( \sharp _{\pi ,\xi }\circ
J^\ast\right) \left( \mathcal{D}_{\alpha
}\beta \right) , \\
&=&\mathcal{\nabla }_{\sharp _{\pi ,\xi }\left( \alpha \right)
}\left( \sharp _{\pi ,\xi }\left( J^\ast\beta \right) \right) +\Phi
\left( \sharp _{\pi
,\xi }\left( \mathcal{D}_{\alpha }\beta \right) \right) , \\
&=&-\mathcal{\nabla }_{\sharp_{\pi,\xi}(\alpha)}\left( \Phi \left( \sharp_{\pi,\xi}(\beta)\right) \right) +\Phi
\left(
\mathcal{\nabla }_{\sharp_{\pi,\xi}(\alpha)}\sharp_{\pi,\xi}(\beta)\right) , \\
&=&-\left( \mathcal{\nabla }_{\sharp_{\pi,\xi}(\alpha)}\Phi \right)(\sharp_{\pi,\xi}(\beta)).
\end{eqnarray*}
\end{proof}

\begin{prop}
Sous les m\^emes hypoth\`eses que dans le lemme ci-dessus, la compatibilit\'e du triplet $(\pi,\xi,g)$ est \'equivalente \`a
$$
(\nabla_X \Phi)(Y)=\dfrac{1}{2}\left(g(\Phi(X),Y)\xi
-\eta(Y)\Phi(X)\right),
$$
pour tous $X,Y\in \mathfrak{X}(M)$, et si de plus la m\'etrique $g$ est d\'efinie positive, alors le triplet $(\pi,\xi,g)$ est compatible si et seulement si la
vari\'et\'e riemannienne presque de contact $(M,\Phi,\xi,\eta,g)$
est $\frac{1}{2}$-Kenmotsu.
\end{prop}
\begin{proof}
Comme on a $J^\ast\flat_g(\xi)=\flat_g(J\xi)=-\flat_g(\Phi \xi)=0$
et
$$
J^\ast\beta(\xi)=J^\ast\beta(\sharp_g(\eta))=\eta(\sharp_g(J^\ast\beta))=\eta(J
\sharp_g(\beta))=-\eta(\Phi(\sharp_g(\beta)))=0, 
$$
alors la formule (\ref{COMPATIBILITE}) devient 
$$
\left( \mathcal{D}_{\alpha }J^\ast\right)
\beta  =\frac{1}{2}\left(
\pi(\alpha,\beta)\eta-
\beta(\xi)J^\ast\alpha \right),
$$
En appliquant $\sharp_{\pi,\xi}$ qui, d'apr\`es la proposition \ref{pi-ass-presque-contact} est une isom\'etrie, donc un isomorphisme, cette derni\`ere formule est \'equivalente \`a 
$$
\sharp_{\pi,\xi}\left(\left( \mathcal{D}_{\alpha }J^\ast\right)
\beta \right) =\frac{1}{2}\left(
\pi(\alpha,\beta)\sharp_{\pi,\xi}(\eta) -
\beta(\xi)\sharp_{\pi,\xi}(J^\ast\alpha) \right). 
$$
Maintenant, d'apr\`es la formule (\ref{f1}),  on a $\sharp_{\pi,\xi}(\eta)=\xi$, et  si on pose $X=\sharp_{\pi,\xi}(\alpha)$ et $Y=\sharp_{\pi,\xi}(\beta) $, alors on a $\beta(\xi)=\eta(Y)$, aussi en utilisant (\ref{pi-xi-J-ast-Phi}), on a  
$\sharp_{\pi,\xi}(J^\ast\alpha)=-\Phi(X)$ et
$$
\pi(\alpha,\beta)
=g(\sharp_g(\alpha),\Phi(\sharp_g(\beta)))=-g(\sharp_g(\alpha),\sharp_g(J^\ast\beta))
=-g^\ast(\alpha,J^\ast\beta) =g(X,\Phi(Y)).
$$ 
Il reste \`a utiliser le lemme ci-dessus. 
\end{proof}

\begin{theo}
Supposons que $(\eta,g)$ est une structure riemannienne de contact sur $M$ et soit
$(\Phi,\xi,\eta,g)$ la structure riemannienne presque de contact
associ\'ee. Supposons que $(\pi,\xi)$ est la structure de Jacobi associ\'ee \`a
la forme de contact $\eta$. Alors le triplet $(\pi,\xi,g)$ est compatible si et seulement si $(M,\Phi,\xi,\eta,g)$ est
$\frac{1}{2}$-Kenmotsu.
\end{theo}
\begin{proof}
Nous avons montr\'e que $\pi$ est bien le champ de bivecteurs de la proposition ci-dessus  et que $\lambda=\eta$, voir la d\'emonstration du th\'eor\`eme 
\ref{riemannienne-contact-levi-civita}.
\end{proof}

\subsection{Vari\'et\'es localement conform\'ement k\"ahl\'eriennes}

Rappelons que si $\omega$ est une $2$-forme non d\'eg\'en\'er\'ee et que $g$ est une m\'etrique riemannienne associ\'ee, la structure presque hermitienne $(\omega,g)$ est hermitienne si la structure presque complexe associ\'ee est int\'egrable, et k\"ahl\'erienne si de plus $\omega$ est ferm\'ee. Rappelons aussi que si $(\omega,g)$ est presque hermitienne, alors elle est k\"ahl\'erienne si et seulement si la $2$-forme $\omega$ est parall\`ele pour la connexion de Levi-Civita de $g$.  

Si $(\omega,\theta)$ est une structure localement conform\'ement symplectique et $(\omega,g)$ une structure
hermitienne, on dit que le triplet $(\omega,\theta,g)$ est une structure localement conform\'ement k\"ahl\'erienne. 

Nous allons montrer que si $(\omega,\theta)$ est une structure localement conform\'ement symplectique sur $M$ et que $(\pi,\xi)$ est la structure de Jacobi associ\'ee, si $g$ est une m\'etrique riemannienne associ\'ee \`a $\omega$ et \`a $(\omega,\theta)$, la compatibilt\'e du triplet $(\pi,\xi,g)$ induit une structure localement
conform\'ement k\"ahl\'erienne sur $M$.

\begin{lem}
Supposons que $\omega\in \Omega^2(M)$ est une $2$-forme non d\'eg\'en\'er\'ee et
soit $\theta\in\Omega^1(M)$. Supposons que $(\pi,\xi)$ est le couple associ\'e
\`a $(\omega,\theta)$. Si la m\'etrique pseudo-riemannienne $g$ est associ\'ee \`a la $2$-forme $\omega$ et au couple $(\omega,\theta)$, alors on a 
$$
J\circ \sharp_{\pi,\xi}=\sharp_{\pi,\xi}\circ J^\ast.
$$ 
\end{lem}
\begin{proof} 
Puisque on a suppos\'e ici que la m\'etrique $g$ est associ\'ee \`a $\omega$ et en utilisant (\ref{omega-pi-xi}), on a  
$$
g(J\sharp_{\pi,\xi}(\alpha),\sharp_{\pi,\xi}(\beta))=\omega(\sharp_{\pi,\xi}(\alpha),\sharp_{\pi,\xi}(\beta))=\pi(\alpha,\beta)=g^\ast(J^\ast\alpha,\beta), 
$$  
et puisque on a suppos\'e que la m\'etrique $g$ est aussi associ\'ee au couple $(\omega,\theta)$, donc que $\sharp_{\pi,\xi}$ est une isom\'etrie, alors 
$$
g(J\sharp_{\pi,\xi}(\alpha),\sharp_{\pi,\xi}(\beta))=g(\sharp_{\pi,\xi}(J^\ast\alpha),\sharp_{\pi,\xi}(\beta)).
$$
Enfin, comme $\sharp_{\pi,\xi}$ est une isom\'etrie, donc un isomorphisme, le r\'esultat en d\'ecoule.  
\end{proof}

\begin{theo}
Supposons que $(\omega,df)$ est une structure conform\'ement symplectique sur $M$ et   que $(\pi,\xi)$ est la structure de Jacobi associ\'ee. Si $g$ est une m\'etrique riemannienne associ\'ee \`a $\omega$ et au couple $(\omega ,df)$. Alors, le triplet $(\pi,\xi,g)$ est compatible si et seulement si le triplet $(\omega,df,g)$ est une structure conform\'ement
k\"ahl\'erienne.
\end{theo}
\begin{proof}
Il s'agit donc de montrer que le triplet $(\pi,\xi,g)$ est compatible si et seulement si le couple $(e^{f}\omega,e^{f}g)$ est compatible, c'est-\`a-dire, si et
seulement si la $2$-forme $e^f\omega$ est parall\`ele pour la
connexion de Levi-Civita $\nabla^f$ associ\'ee \`a la m\'etrique
$g^f=e^fg$. Comme les connexions $\nabla$ et $\nabla^f$ sont reli\'ees  par la formule  
$$
\nabla _{X}^{f}Y=\nabla _{X}Y+\frac{1}{2}\left(
X(f)Y+Y(f)X-g(X,Y)\textrm{grad}_g f \right), 
$$
o\`u $\textrm{grad}_gf=\sharp_g(df)$, il vient que 
$$
\begin{array}{lll}
\nabla^f\omega(X,Y,Z) & = & \nabla\omega(X,Y,Z)-X(f)\omega(Y,Z)-\frac{1}{2}Y(f)\omega(X,Z)+\frac{1}{2}Z(f)\omega(X,Y)\\
 &   & +\frac{1}{2}\left(g(X,Y)\omega(\textrm{grad}_g f,Z)-g(X,Z)\omega(\textrm{grad}_g
 f,Y)\right),
\end{array}
$$
et donc que 
$$
\nabla^f (e^f\omega)(X,Y,Z)= e^f\left(X(f)\omega(Y,Z)+\nabla^f\omega(X,Y,Z)\right) =e^f \Lambda_f(X,Y,Z), 
$$
o\`u on a pos\'e 
$$
\begin{array}{lll}
\Lambda_f(X,Y,Z) &=& \nabla\omega(X,Y,Z)-\dfrac{1}{2}\left(Y(f)\omega(X,Z)-Z(f)\omega(X,Y)\right) \\ 
& & +\frac{1}{2}\left(g(X,Y)\omega(\textrm{grad}_g f,Z)-g(X,Z)\omega(\textrm{grad}_g f,Y)\right).
 \end{array}
$$
On en d\'eduit que  $\nabla^f(e^f\omega)=0$ si et seulement si $\Lambda_f=0$, donc que le couple $(e^f\omega,e^fg)$ est compatible si et seulement si 
$$
\nabla\omega(X,Y,Z)\!=\!\dfrac{1}{2}\left(Y(f)\omega(X,\!Z)\!-\!Z(f)\omega(X,\!Y)
 \!-\!g(X,\!Y)\omega(\textrm{grad}_g f,\!Z)\!+\!g(X,\!Z)\omega(\textrm{grad}_g f,\!Y)\right).
$$ 
Montrons maintenant que cette derni\`ere identit\'e est \'equivalente 
\`a la formule (\ref{COMPATIBILITE2}). Soient $\alpha ,\beta, \gamma \in \Omega ^{1}(M)$ tels que 
$X=\sharp_{\pi,\xi}(\alpha)$, $Y=\sharp_{\pi,\xi}(\beta)$ et
$Z=\sharp_{\pi,\xi}(\gamma)$. D'une part, d'apr\`es le corollaire \ref{levi-civita-omega}, on a $\nabla \omega(X,Y,Z)=\mathcal{D}\pi(\alpha,\beta,\gamma)$. D'autre part, en posant $\theta=df$, on a $
Y(f)=\theta(Y)=\theta(\sharp_\pi(\beta))+\beta(\xi)\theta(\xi)=-\beta(\sharp_\pi(\theta))=-\beta(\xi)$ et de m\^eme  $Z(f)=-\gamma(\xi)$.  
Aussi, d'apr\`es (\ref{omega-pi-xi})
, on a $\omega(X,Y)=\pi(\alpha,\beta)$ et $\omega(X,Z)=\pi(\alpha,\gamma)$. 
Enfin, comme $g$ est une m\'etrique associ\'ee \`a $\omega$, il vient que 
$$
\omega(\textrm{grad}_g f, Y) =-\omega(Y,\sharp_g(\theta))=-g(JY,\sharp_g(\theta))=-\theta(JY)=\omega(\xi,JY)=\omega(\sharp_{\pi,\xi}(\theta),J\sharp_{\pi,\xi}(\beta)),
$$
et puisque $g$ est associ\'ee \`a $\omega$ et \`a $(\omega,\theta)$, en utilisant le lemme ci-dessus et (\ref{omega-pi-xi}), on obtient 
$$
\omega(\textrm{grad}_g f, Y)=\omega(\sharp_{\pi,\xi}(\theta),\sharp_{\pi,\xi}(J^\ast\beta))=\pi(\theta,J^\ast\beta)=J^\ast\beta(\sharp_\pi(\theta))=J^\ast\beta(\xi)
$$
et de m\^eme $\omega(\textrm{grad}_g f, Z)=J^\ast\gamma(\xi)$.
\end{proof}

 \bigskip
Y. A\"{\i}t Amrane, Laboratoire Alg\`ebre et Th\'eorie des
Nombres,\\
Facult\'e de Math\'ematiques,\\
USTHB, BP 32, El-Alia, 16111 Bab-Ezzouar, Alger, Alg\'erie. \\
e-mail : yacinait@gmail.com \bigskip\\
A. Zeglaoui, Laboratoire Alg\`ebre et Th\'eorie des
Nombres,\\
Facult\'e de Math\'ematiques,\\
USTHB, BP 32, El-Alia, 16111 Bab-Ezzouar, Alger, Alg\'erie. \\
e-mail : ahmed.zeglaoui@gmail.com \\

\end{document}